\documentclass[a4paper,12pt]{article}
\usepackage[centertags]{amsmath}
\usepackage{amsfonts}
\usepackage{amssymb}
\usepackage{amsthm}
\usepackage{amsmath}
\usepackage{dsfont}
\usepackage{multirow}
\usepackage{graphicx}
\usepackage{tikz, subfigure}
 \usepackage{authblk}
\usetikzlibrary{scopes,arrows,positioning,topaths,calc,fadings}
\usepackage{pstricks-add}
\usepackage{verbatim}

\usepackage[latin1]{inputenc}
\usepackage{tikz}
\usepackage[textwidth=17cm]{geometry}
\usetikzlibrary{arrows}
\usetikzlibrary {positioning}
\usetikzlibrary{calc}
\definecolor {processblue}{cmyk}{0.96,0,0,0}
\usepackage{amsfonts,graphicx,amsmath,amssymb,hyperref,color}
\usepackage[english]{babel}

\newtheorem{Theorem}{Theorem}[section]
\newtheorem{Proposition}{Proposition}[section]
\newtheorem{Lemma}{Lemma}[section]
\newtheorem{Corollary}{Corollary}[section]
\newtheorem{Open problem}{Open problem}[section]

\newtheorem{Definition}{Definition}[section]

\newtheorem{Example}{Example}[section]

\newtheorem{Remark}{Remark}[section]

\AtEndDocument{\bigskip{\footnotesize%
  \textsc{Department of Mathematics, Utrecht University, Fac Wiskunde en informatica and MRI, Budapestlaan 6, P.O. Box 80.000, 3508 TA Utrecht, The Netherlands} \par
  \textit{E-mail address:} \texttt{K.Jiang1@uu.nl} \par
}}

\usepackage{lipsum}

\newcounter{tmp}

\title{xxxx}
\date{}
 \title{Hausdorff dimension of  the arithmetic sum of self-similar sets}
\author{Kan Jiang}
\begin{document}
\maketitle
\begin{abstract}
Let $\beta>1$. We define a class of similitudes
\[S:=\left\{f_{i}(x)=\dfrac{x}{\beta^{n_i}}+a_i:n_i\in \mathbb{N}^{+}, a_i\in \mathbb{R}\right\}.\]
Taking any finite collection of similitudes $\{f_{i}(x)\}_{i=1}^{m} $ from $S$, it is well known that there is a  unique self-similar set $K_1$ satisfying $K_1=\cup_{i=1}^{m} f_{i}(K_1)$.  Similarly,   another self-similar set $K_2$ can be generated via the finite contractive maps of $S$. We call $K_1+K_2=\{x+y:x\in K_1, y\in K_2\}$ the arithmetic sum of two  self-similar sets. In this paper, we prove that $K_1+K_2$ is either a self-similar set or a unique attractor of some infinite iterated function system.  Using this result  we  can calculate the exact Hausdorff dimension of $K_1+K_2$ under some conditions, which partially  provides the dimensional result of $K_1+K_2$ if the IFS's of $K_1$ and $K_2$ fail the irrationality assumption, see  Peres and Shmerkin \cite{PS}.
\end{abstract}

\section{Introduction}
Let $\{g_{j}\}_{j=1}^{m}$ be an iterated function system (IFS) of similitudes which are defined on $\mathbb{R}$ by
\[g_j(x)=r_jx+a_j,\]
where the similarity ratios satisfy $0<r_j<1$ and the translation parameter $a_j\in \mathbb{R}$. It is well known that there exists a unique non-empty compact set $K\subset \mathbb{R}$ such that
\begin{equation}
\label{Hutchinson formula}
K=\bigcup_{j=1}^{m}g_j(K).
\end{equation}
We call $K$ the self-similar set or attractor for the IFS
$\{g_{j}\}_{j=1}^{m}$, see \cite{Hutchinson} for further details. The IFS $\{g_{j}\}_{j=1}^{m}$ is called homogeneous if all the similarity ratios are equal.
We say that $\{g_{j}\}_{j=1}^{m}$ satisfies the open set condition(OSC) \cite{Hutchinson}  if there exists a non-empty bounded open set $V\subseteq \mathbb{R}$ such that
\[g_i(V)\cap g_{j}(V)=\varnothing,\, i\neq j\]
and $g_j(V)\subseteq V$ for all  $1\leq j\leq m$.
Under the open set condition, the Hausdorff dimension of $K$ coincides with  the similarity dimension which is the unique solution $s$ of the equation $\sum_{j=1}^{m}r_j^s=1$.

 Let $F_1$ and $F_2$ be the self-similar sets with IFS's $\{r_ix+a_i\}^{n}_{i=1}$ and $\{r_j^{'}x+b^{'}_j\}^{m}_{j=1}$ respectively.
We call $F_1+F_2=\{x+y:x\in F_1, y\in F_2\}$ the arithmetic sum of self-similar sets.  The arithmetic sum of Cantor sets appears naturally in  dynamical systems. Palis \cite{Palis} posed the following problem which is currently known as the Palis' conjecture. Whether it is true (at least generically) that the arithmetic  sum of dynamically defined Cantor sets either has measure zero or contains an interval. This conjecture was solved in \cite{Yoccoz}. However, for the general self-similar sets this conjecture is still open.  In \cite{MO}, Mendes and Oliveira proved  that for the homogeneous Cantor sets, there are five possible structures for the sum. For the fractal structure, i.e. the similarity of the sum of self-similar sets, there are few results regarding this aspect. This is the first reason why we study the sum of self-similar sets.
 Another natural question  concerning  the sum of self-similar sets is to consider the Hausdorff dimension  or  Hausdorff measure of $F_1+F_2$. 
 Many papers have been devoted to this aspect. Let $C_a$ be the central Cantor set generated by removing a central interval of length $1-2a$ from $[0,1]$, and then continuing this process inductively on each remaining two intervals. Denote $\gamma(a)=\dim_{H}(C_a)=\dfrac{\log2}{-\log a}$. Peres and Solomyak  \cite{PeresSolomyak1} proved that
 \begin{Theorem}
 Given a fixed compact set $K\subset \mathbb{R}, $ the following two statements hold for almost every $a\in(0, \dfrac{1}{2})$:
 
 if $\gamma(a)+\dim_{H}(K)\leq 1$, then $\dim_{H}(K+C_a)=\gamma(a)+\dim_{H}(K)$; 
 
 if $\gamma(a)+\dim_{H}(K)> 1$, then the Lebesgue measure of $C_a+K$ is positive. 
 \end{Theorem}
 Motivated by this result Eroglu \cite{Eroglu} considered the Hausdorff measure of the arithmetic sum of two Cantor sets, and gave a necessary and sufficient condition such that the Hausdorff measure of  the sum of  Cantor sets is positive.
 Peres and Solomyak's main idea is using the potential theory. This is the main reason why their result is the almost-type result. An important progress of the dimensional problem is due to Peres and Shmerkin. 
 In \cite{PS}, Peres and Shmerkin showed that
 \begin{Theorem}
  If there exist $i,\,j$ satisfying $\dfrac{\log r_i}{\log r^{'}_j}\notin \mathbb{Q}$, then \[\dim_{H}(F_1+F_2)=\min\{1,\dim_{H}(F_1)+\dim_{H}(F_2)\}.\]
  \end{Theorem}
The hypothesis of this theorem is called the  irrationality assumption. 
It is easy to see that many  pairs of  iterated function systems satisfy this assumption. Peres and Shmerkin's formula gives a sufficient condition under which the expected dimension of the sum of self-similar sets can be obtained.  Their main idea is to project the product of two one-dimensional self-similar sets into the real line and to show that under the irrationality assumption the expected dimension of $F_1+F_2$ can be achieved. 
Later, Nazarov   et.al. \cite{Nazarov} investigated  similar problem for the convolutions of Cantor measures without resonance. 

Motivated by Peres and Shmerkin's result and Palis' conjecture,  we consider the IFS's of $F_1$ and $F_2$ failing  the irrationality assumption. With a little effort, it can be shown that  the IFS's $\{r_ix+a_i\}^{n}_{i=1}$ and $\{r_j^{'}x+b^{'}_j\}^{m}_{j=1}$ do not satisfy the irrationality assumption if and only if there exist $\beta>1$, $n_i$ and $m_j \in \mathbb{N}$ such that $r_i=\dfrac{1}{\beta^{n_i}}$,  $1\leq i \leq n $ and $r^{'}_j=\dfrac{1}{\beta^{m_j}}$,  $1\leq j \leq m $. Unless stated otherwise, in what follows we always assume that the similitudes of $K_1$ and $K_2$ are  from \[S:=\left\{f_{i}(x)=\dfrac{x}{\beta^{n_i}}+a_i:n_i\in \mathbb{N}^{+}, a_i\in \mathbb{R}\right\}.\]  We suppose  without loss of generality that   the IFS's of $K_1$ and $K_2$ are 
$\{f_{i}(x)=\frac{x}{\beta^{n_i}}+a_i\}_{i=1}^{n}$ and
  $\{g_{j}(x)=\frac{x}{\beta^{m_j}}+b_j\}_{j=1}^{m}$, respectively.  

We shall prove that $K_1+K_2$ is either a self-similar set or an attractor of some infinite iterated function system (IIFS) \cite{MRD, HF}. Therefore, calculating the Hausdorff dimension of $K_1+K_2$ is reduced to considering the dimension of the attractor of some IFS (IIFS). 
It is well known that generally it is difficult to calculate the Hausdorff dimension of  a self-similar set, especially when  overlaps occur.  It is  much more difficult to find the dimension of the attractor of some IIFS even if the IIFS satisfies certain separation condition. Here the attractor of the IIFS is in the sense of Definition \ref{IIIFS}, we will introduce this definition in the next section. In fact,  Peres and Shmerkin's dimensional formula   implies that 
we may not find the exact  Hausdorff dimension of $F_1+F_2$ generally. In this paper, we  shall consider some cases  which allow us to calculate the dimension of $K_1+K_2$ explicitly. 
An important difference between our main result and Peres and Shmerkin's formula is that we may not obtain the expected dimension for the sum of self-similar sets, see the first example in section 4.
 Peres and Shmerkin gave a uniform formula while we emphasize on the individual example. In other words, our method is analyzing single example rather than giving a uniform formula for the dimension of the sum of self-similar sets. When $K_1+K_2$ is a self-similar set with overlapping IFS, the techniques of the paper \cite{Hochman} could be useful.  However, this is beyond our discussion, and we do not give further details. 
 
For the topological structure of $K_1+K_2$, e.g. connected  property and so on,  generally we may not   easily  get further information. The  main reasons are that the IFS (IIFS) of $K_1+K_2$ may vary from each other and that discussing these two cases needs different techniques. 
 
The  structure of the paper is as follows. In section 2,  we introduce some basic results of   infinite iterated function systems and define some necessary terminology.  Next, we prove the similarity of $K_1+K_2$. 
In section 3, we concentrate on the Hausdorff dimension of $K_1+K_2$.  We consider both cases, i.e. $K_1+K_2$ is a self-similar set or a unique attractor of some IIFS, and give some dimensional results. 
 In section 4, we offer some examples for which we can explicitly calculate the  Hausdorff dimension of  $K_1+K_2$. Finally, we give some further remarks.

\section{Preliminaries and Main results}
\subsection{Infinite iterated function systems}
Before stating our main results, we introduce some  definitions and  results  of infinite iterated function systems (IIFS).
Infinite iterated function systems behave differently from IFS's \cite{MRD}, \cite{HF}. There are two definitions of the invariant set of  IIFS, see for example, \cite{HF},\, \cite{MRD} and \cite{HM}. We adopt Fernau's definition \cite{HF}.
\begin{Definition}\label{IIIFS}
Let $\mathcal{A}=\{\phi_{i}(x)=r_ix+a_i: i\in \mathbb{N},\,0<r_i<1, a_i\in \mathbb{R}\}$. If there exists $0<s<1$ such that for
every $\phi_{i}\in \mathcal{A}$, $|\phi_{i}(x)-\phi_{i}(y)|\leq s
|x-y| $, then $\mathcal{A}$ is called an infinite iterated function system, abbreviated as IIFS. A  unique non-empty compact  set $J$  is called  the attractor of $\mathcal{A}$ if
\[J=\overline{\bigcup_{i\in \mathbb{N}}\phi_{i}(J)},\]
where $ \overline{A}$ denotes the closure of $A$.
\end{Definition}
\setcounter{Remark}{1}
\begin{Remark}
 The existence and uniqueness of $J$ can be found in  \cite{HF}. In \cite{MRD}, Mauldin and Urbanski gave another definition of the attractor of IIFS, i.e. $J_{0}=\bigcup_{i\in \mathbb{N}}\phi_{i}(J_{0})$. However, for their definition the attractor $J_{0}$ may not be unique  or compact, see  example 1.3 from \cite{HF}. Evidently, $\overline{J_0}=J$.
\end{Remark}
An infinite iterated function system $\mathcal{A}=\{\phi_{i}: i\in \mathbb{N}\}$ satisfies the open set condition if there exists  a non-empty bounded open set $O\subseteq \mathbb{R}$ such that
\[\phi_i(O)\cap \phi_{j}(O)=\varnothing,\, i\neq j,\]
and $\phi_j(O)\subseteq O$ for all  $j \in \mathbb{N}$.
Under this separation condition, we can find the Hausdorff dimension of $J_0$. The following result can be found in \cite{MRD}, \cite{M} or \cite{HM}.
\setcounter{Theorem}{2}
\begin{Theorem}\label{DimensionIIFS}
For any IIFS satisfying the open set condition, we have
\[
\dim_{H}(J_0)=\inf\left\{t:\sum_{i\in \mathbb{N}}r_i^{t}\leq 1\right\}.
\]
\end{Theorem}
On the other hand, generally the Hausdorff dimension of $J$ is more complicated.  One of the difficulties  is to analyze  $J\setminus J_0$, see \cite[Corollary 2]{HM}. For  the most cases, we shall prove that $J=K_1+K_2$ is an attractor of some IIFS in the sense of Definition \ref{IIIFS}. This makes the dimension of $K_1+K_2$ complicated.  We mentioned above that $\overline{J_0}=J$.  If $J_0$ and $J$ coincide except for  a countable set, then by the countable stability of  the Hausdorff dimension we have that $\dim_{H}(J_0)=\dim_{H}(J)$. We will give a sufficient  condition  under  which  we can identify  $J_0$  with $J$  apart from a countable set. This is the main idea  we will implement,  provided  $K_1+K_2$ is the unique attractor of some IIFS.
\subsection{Some definitions}
In this section, we introduce some definitions which  make our discussion  far more succinct.  Given any finite  reals $s_1,s_2,s_3,\cdots,  s_n$. Let $\sum=\{s_1,\,s_2,\,\cdots,\,s_n\}^{\mathbb{N}}$ be a symbolic space.
We say  $c_1c_2\cdots c_m \in\{s_1,\,s_2,\,\cdots,\,s_n\}^{m}$ is a block with length $m$, and we  use  capital letters with hats to denote the finite blocks of $\Sigma$. For instance, we denote $c_1c_2\cdots c_m$ by $\hat{P}$, i.e.  $\hat{P}=c_1c_2\cdots c_m$. 
\setcounter{tmp}{\value{Theorem}}
\setcounter{Definition}{3}
\begin{Definition}
 Let $\hat{P}_{1}=d_1d_2\cdots d_m$ and $ \hat{P}_{2}=c_1c_2\cdots c_m$ be two blocks of $\{s_1,\,s_2,\,\cdots,\,s_n\}^{m}$. We define the concatenation of $\hat{P}_{1}$ and $\hat{P}_{2}$ by   $\hat{P}_{1}*\hat{P}_{2}=d_1d_2\cdots d_mc_1c_2\cdots c_m$. The sum of $\hat{P}_{1}$ and $\hat{P}_{2}$ is defined by $\hat{P}_{1}+\hat{P}_{2}=(d_1+c_1)(d_2+c_2)\cdots (d_m+c_m)$. Concatenating  $k\in \mathbb{N}$ blocks of $\hat{P}_{1}$ is denoted by  \[\hat{P}_{1}^{k}=\underbrace{\hat{P}_{1}*\hat{P}_{1}*\cdots*\hat{P}_{1}}_\text{$k$ times}.\] The value of the block
$\hat{P}_{1}=d_1d_2\cdots d_m$ with respect to $\beta>1$ is \[(d_1d_2\cdots d_m)_{\beta}= \dfrac{d_1}{\beta}+\dfrac{d_2}{\beta^2}+\cdots+\dfrac{d_m}{\beta^m}.\]
Similarly, we can define the value of an infinite sequence $(d_n)\in \sum$ by  $(d_n)_{\beta}=\sum\limits_{n=1}^{\infty}\dfrac{d_n}{\beta^n}$.
\end{Definition}
\setcounter{Remark}{4}
\begin{Remark}\label{infinitesum}
In this definition, when we define the summation of two blocks, we assume that these two blocks have the same length. However,  in some cases we may need to consider the concatenation of infinite blocks. For instance, let $ \{\hat{P}_{i}\}^{\infty}_{i=1}$ and $\{\hat{Q}_{i}\}^{\infty}_{i=1}$ be two block sets, the concatenations of $\hat{P}_{1}\ast \hat{P}_{2}\ast\cdots$ and $\hat{Q}_{1}\ast \hat{Q}_{2}\ast\cdots$ are  two infinite sequences in $\Sigma$, we denote them  by $(a_n)$ and $(b_n)$ respectively. The summation of $\hat{P}_{1}\ast \hat{P}_{2}\ast\cdots$ and $\hat{Q}_{1}\ast \hat{Q}_{2}\ast\cdots$ is $(a_n+b_n)_{n=1}^{\infty}$. 
We shall emphasize this case in the proofs of some results. 
\end{Remark}
Now we give the definition of the codings of  the points in the self-similar sets. It is sightly  different from the usual way. 
Recall the  IFS's of $K_1$ and $K_2$ are
$\{f_{i}(x)=\frac{x}{\beta^{n_i}}+a_i\}_{i=1}^{n}$ and
  $\{g_{j}(x)=\frac{x}{\beta^{m_j}}+b_j\}_{j=1}^{m}$, where $n_i, a_i, m_j, b_j$ are determined by the IFS's of $K_1$ and $K_2$. It is well known that for any $x\in K_1$,  there exists $(i_k)_{k=1}^{\infty}$ such that \[x=\lim\limits_{k\to \infty}f_{i_1}\circ f_{i_2} \circ\cdots\circ f_{i_k}(0).\]
Usually, $(i_k)_{k=1}^{\infty}$ is called a coding of $x$. Nevertheless, we may make use of another representation.

 Note
that
\[f_{i}(x)=\frac{x}{\beta^{n_i}}+a_i=\frac{x+\beta^{n_i}a_i}{\beta^{n_i}}=\frac{x}{\beta^{n_i}}+\frac{0}{\beta}+\frac{0}{\beta^2}+\cdots+
\frac{0}{\beta^{n_i-1}}+\frac{\beta^{n_i}a_i}{\beta^{n_i}},\] therefore, we can  identify  $f_i(x)$ with a block $(\underbrace{000\cdots 0}_{n_i-1}a_i^{'}
)$,\, where $a_i^{'}=\beta^{n_i}a_i$. In fact, $f_i(x)$  and $(\underbrace{000\cdots 0}_{n_i-1}a_i^{'}
)$ can be  determined mutually.  Given  $(\underbrace{000\cdots 0}_{n_i-1}a_i^{'}
)$ with length $n_i$ and $a_i^{'}=\beta^{n_i}a_i$, we can find a similitude \[f_{i}(x)=\frac{x}{\beta^{n_i}}+\frac{0}{\beta}+\frac{0}{\beta^2}+\cdots+
\frac{0}{\beta^{n_i-1}}+\frac{\beta^{n_i}a_i}{\beta^{n_i}}=\frac{x+\beta^{n_i}a_i}{\beta^{n_i}}=\frac{x}{\beta^{n_i}}+a_i.\]   For simplicity  we denote this block by
$\hat{P}_{i}=(\underbrace{000\cdots 0}_{n_i-1}a_i^{'}
)$ if there is no fear of ambiguity.  We identify  $f_{i}$ with $f_{\hat{P}_{i}}$. The only difference between $f_{i}$ and $f_{\hat{P}_{i}}$ is the symbol as both of them represent the map $f_{i}(x)=f_{\hat{P}_{i}}(x)=\frac{x}{\beta^{n_i}}+a_i$. Similarly, we may define blocks in terms of the IFS of $K_2$. Let $D_1=\{\hat{P}_{1},\, \hat{P}_{2},\,\cdots,\,\hat{P}_{n}\}$ and $D_2=\{\hat{Q}_{1},\, \hat{Q}_{2},\,\cdots,\,\hat{Q}_{m}\}$, where $\hat{P}_{i}=(\underbrace{000\cdots 0}_{n_i-1}a_i^{'}
)$, $a_i^{'}=\beta^{n_i}a_i$,  $\hat{Q}_{j}=(\underbrace{000\cdots 0}_{m_j-1}b_j^{'}
)$ and $b_j^{'}=\beta^{m_j}b_j$. We say  $D_1$ and $D_2$ are the digit sets of $K_1$ and $K_2$ respectively.   The elements of $D_i$ are called the blocks. We emphasize that 
different blocks may stand for  the same similitude, for example let  $\hat{R_1}=(08)$ and  $\hat{R_2}=(22)$ be two blocks  with respect to base 3,  since their associated similitudes coincide, i.e.  $\varphi_{\hat{R}}(x)=\dfrac{x}{3^2}+\dfrac{0}{3}+\dfrac{8}{3^2}=\dfrac{x}{3^2}+\dfrac{2}{3}+\dfrac{2}{3^2}$, we can choose either of them if we want to find the digit sets of $K_i, 1\leq i\leq 2$. This replacement does not affect our main result. Usually, we pick the simpler blocks which  facilitate  our calculation. Once we choose the blocks, we fix them. With this new representation,  we have  following simple lemma.
\setcounter{Lemma}{5}
\begin{Lemma}\label{coding}
\[K_{1}=\{x=\lim\limits_{n\to \infty}f_{\hat{P}_{i_1}}\circ f_{\hat{P}_{i_2}} \circ\cdots\circ f_{\hat{P}_{i_n}}(0):\hat{P}_{i_j}\in D_1\}.\]
\[K_{2}=\{y=\lim\limits_{n\to \infty}g_{\hat{Q}_{i_1}}\circ g_{\hat{Q}_{i_2}} \circ\cdots\circ g_{\hat{Q}_{i_n}}(0):\hat{Q}_{i_j}\in D_2\}.\]
We call  the concatenation $\hat{P}_{i_1}\ast \hat{P}_{i_2} \ast \cdots$ ($\hat{Q}_{i_1}\ast \hat{Q}_{i_2} \ast \cdots$) a coding of x ($y$). 
\end{Lemma}
\begin{proof}
 For any $x\in K_1$,  we know that there exists $(i_n)_{n=1}^{\infty}$ such that \[x=\lim\limits_{n\to \infty}f_{i_1}\circ f_{i_2} \circ\cdots\circ f_{i_n}(0).\] The lemma is a restatement of this fact.
\end{proof}
\setcounter{Remark}{6}
\begin{Remark}
Although the lemma above is very simple, the significance of this lemma is that we can translate over the problem, i.e. in order to study the sum of two numbers from $K_1$ and $K_2$ respectively, it is sufficient to  consider the sum of the blocks from $D_1$ and $D_2$.
\end{Remark}
Motivated by this lemma, we define a crucial
definition of this paper.
\setcounter{Definition}{7}
\begin{Definition}\label{Matching}
Take $s$ blocks
\[\hat{P}_{i_1},\,\hat{P}_{i_2},\,\hat{P}_{i_3},\,\cdots,\,
\hat{P}_{i_{s}}\] from $D_1$ with  lengths
$p_{1},\,p_{2},\,p_{3},\,\cdots,\, p_{s}$,\, $t$
blocks \[\hat{Q}_{j_1},\,\hat{Q}_{j_2},\,\hat{Q}_{j_3},\,\cdots,\,
\hat{Q}_{j_{t}}\] from $D_2$ with lengths
$q_{1},\,q_{2},\,q_{3},\,\cdots ,\,q_{t}$. If there exist
integers
 $k_{1},\,k_{2},\,k_{3},\cdots,
k_{s}$,\\ $l_{1},\,l_{2},\,l_{3},\cdots, l_{t}$ such that
\[
\sum_{i=1}^{s}k_{i}p_{i}=\sum_{j=1}^{t}l_{i}q_{i},
\]
then the block $(\hat{P}_{i_1}^{k_1}\ast \hat{P}_{i_2}^{k_2}\ast\cdots\ast
\hat{P}_{i_{s}}^{k_s})+(\hat{Q}_{j_1}^{l_1}\ast
\hat{Q}_{j_2}^{l_2}\ast\cdots\ast\hat{Q}_{j_{t}}^{l_t})$ is
called a Matching   with respect to $\beta$.
\end{Definition}
\setcounter{Remark}{8}
\begin{Remark}\label{sum}
Let $A$ and $B$ be two concatenations of some blocks from $D_1$ and $D_2$  respectively. If $A$ and $B$ have the same length, then the summation of $A$ and $B$ is a Matching, i.e. $A+B$ is a Matching. 
 We call the elements of $D_i$ blocks. However,  a Matching, in fact,  is also a block which is the sum of  concatenated blocks  from $D_1$ and $D_2$ respectively.  In what follows, we still call a Matching a block if there is no fear of ambiguity.  Clearly, in this definition the blocks $\hat{P}_{i_j}$ and $\hat{P}_{i_k}$ ($j\neq k$) could coincide. 
Given a  Matching we may find its associated similitude. For instance, let
 $(abc)$ be  a Matching with respect to $\beta$, then the corresponding
 similitude is
$\varphi(x)=\frac{x}{\beta^{3}}+\frac{a}{\beta}+\frac{b}{\beta^2}+
\frac{c}{\beta^{3}}$.
\end{Remark}

We  show that $D_1$ and $D_2$  generate  countably many Matchings.
\setcounter{Lemma}{9}
\begin{Lemma}\label{GenerateMatchings}
The cardinality of Matchings which are generated by $D_1$ and $D_2$ is at most countable.
\end{Lemma}

\begin{proof}
The proof is constructive. Firstly,  we find out all the  possible Matchings which have length 1.  The cardinality of Matchings with length $1$ is finite due to the finite cardinalities of $D_1$ and $D_2$. If there are no such Matchings (see Example \ref{EX3}), we then consider the Matchings with length $2$. Similarly,  we can find finite Matchings which are of length $2$.  If there do not exist  such Matchings, then we
may consider the Matchings with length 3.
We continue this procedure and prove the lemma. However, the following adjustment is helpful to reduce some unnecessary Matchings,  i.e.  if the new born
Matchings can be concatenated by the old Matchings,
then we do not choose these new Matchings. In the remaining paper we always abide by this rule. In some cases, after
some steps, all the new Matchings can be concatenated by the former
old Matchings (see Example \ref{EX2}), then we stop the procedure. For this 
case, the cardinality of Matchings is finite. If the procedure can be continued for infinitely many times, then the  cardinality of Matchings is infinitely countable. Hence, the cardinality of Matchings is either finite or countably infinite. 
\end{proof}
\setcounter{Remark}{10}
\begin{Remark}
We shall prove that if the cardinality of Matchings  is finite, then $K_1+K_2$ is a self-similar set while $K_1+K_2$ is  the unique attractor of some IIFS if   the cardinality of Matchings is infinitely countable.
\end{Remark}
\setcounter{Example}{11}
\begin{Example}\label{EX1}
 Let $K_1=K_2$ be the attractor of the IFS  $\{g_{1}(x)=\frac{x}{3}
,\,g_{2}(x)=\frac{x+8}{9}\}$. All the possible Matchings are
\[\{(0),\,(22),\,(44),\, (242),\,(2442),\,(24442),\,(244442),\,(2444442),\,(24444442),\cdots\},\]
where $D_1=D_2=\{(0),\,(08)=(22)\}$. Here, for simplicity we assume that $\hat{R}=(08)=(22)$ as their corresponding similitudes are the same, i.e. $\varphi_{\hat{R}}(x)=\dfrac{x}{3^2}+\dfrac{0}{3}+\dfrac{8}{3^2}=\dfrac{x}{3^2}+\dfrac{2}{3}+\dfrac{2}{3^2}$.
\end{Example}
\begin{Example}\label{EX2}
Let $\{f_{1}(x)=\frac{x}{3}
,\,f_{2}(x)=\frac{x+2}{3}\}$ be the IFS of  $K_1$,  $K_2$
 is generated by $\{g_{1}(x)=\frac{x}{3} ,\,g_{2}(x)=\frac{x+8}{9}\}$.  Then  the Matchings generated by $D_1$ and $D_2$ are $\{(0),\,(2),\, (24),\, (42),\,(44)\}$, where $D_1=\{(0),\,(2)\}$ and $D_{2}=\{(0),\,(22)\}$.
\end{Example}
\begin{Example}\label{EX3}
Let $\{f_{1}(x)=\frac{x}{9}
,\,f_{2}(x)=\frac{x}{3^3}+\frac{2}{3}+\frac{2}{3^2}+\frac{2}{3^3}\}$ be the IFS of  $K_1=K_2$, where $D_1=D_{2}=\{(00),\,(222)\}$. For this example, there is no Matching with length $1.$ 
\end{Example}
After  we find  all the possible Matchings, we denote this set by
\[D=\{\hat{R}_{1} ,\,\hat{R}_{2},\,\cdots,\,\hat{R}_{n-1},\,
\hat{R}_{n},\cdots\}, \] the lengths of these Matchings are increasing. By
Remark \ref{sum}, $D$ uniquely determines a set of similitudes
$\Phi^{\infty}\triangleq\{\phi_1,\,\phi_2,\,\phi_3,\,\phi_4,\cdots\}
$.  We define
$E\triangleq
\bigcup\limits_{\{\phi_n\}\in\Phi^{\infty}}\bigcap\limits_{n=1}^{\infty}\phi_{1}\circ\phi_{2}\cdots\circ
\phi_{n}([0,1])$ and have
$E=\bigcup\limits_{i\in \mathbb{N}}\phi_{i}(E)$, see  section 2  from \cite{MRD}.

Now we state the first main result. 
\setcounter{Theorem}{14}
\begin{Theorem}\label{Structure}
$K_1+K_2$ is  either a self-similar set or an attractor of some infinite iterated function system. More precisely, if the cardinality of Matchings  is finite, then  $K_1+K_2$ is a self-similar set. When the cardinality is  infinitely countable, we have \[K_1+K_2=\overline{\bigcup_{\phi_{i}\in \Phi^{\infty}}\phi_{i}(K_1+K_2)}.\]
\end{Theorem}
\setcounter{Remark}{15}
\begin{Remark}
 A minor modification enables us to prove  the following stronger  result:
for any  $n\in \mathbb{N}^{+}$ and any $\{K_i\}^{n}_{i=1}$,
$K_1+K_2+\cdots+ K_n=\{\sum_{i=1}^{n}x_i:x_i\in K_i\}$  is either a self-similar set or a unique attractor of some IIFS, where $\{K_i\}^{n}_{i=1}$ are generated by the similitudes of $S$.
In \cite{MO}, Mendes and Oliveira proved  that for the homogeneous Cantor sets, there are five possible structures for the sum. However, in our setting we may  find  only two structures, i.e.  $K_1+K_2$ is either a self-similar set or an attractor of some IIFS. 
\end{Remark}
We have an interesting corollary of Theorem \ref{Structure}.
\setcounter{Corollary}{16}
\begin{Corollary}\label{PB}
 Let $F_1$ and $F_2$ be the self-similar sets with IFS's $\{r_ix+a_i\}^{n}_{i=1}$ and $\{r_j^{'}x+b^{'}_j\}^{m}_{j=1}$, if $0<r_i, r^{'}_j<1$ for any $ 1\leq i \leq n $ and  $ 1\leq j \leq m$, then
\[\dim_{P}(F_1+F_2)=\overline{\dim}_{B}(F_1+F_2).\]
\end{Corollary}

\subsection{Proofs of Theorem \ref{Structure} and Corollary \ref{PB}}
To begin with we assume that the cardinality of all Matchings is infinitely countable.
Before we prove the main results, we need some preliminaries.  In Lemma \ref{coding} we give the definition of the codings of $K_i$, $1\leq i\leq 2$.  Here we define the coding of $x+y\in K_1+K_2$ in a natural way, i.e.
 we denote the coding of $x+y$ by $(x_n+y_n)_{n=1}^{\infty}$, where $(x_n)$ and $(y_n)$ are the codings of $x$ and  $y$ respectively. 

We know that $(x_n)$ ($(y_n)$) can be decomposed into infinite blocks from $D_1$($D_2$), see the following figure

\begin{center}
\begin{tabular}{ |p{1cm}  p{1cm}  p{1cm}  p{1cm} p{1cm}  p{1cm} p{1cm}  p{1cm} p{1cm}  p{1cm} p{1cm}  p{1cm} p{1cm}  p{1cm} p{1cm}  p{1cm} p{1cm}  p{1cm}  p{1cm}  p{1cm} p{1cm}  p{1cm} p{1cm}  p{1cm} p{1cm}  p{1cm} p{1cm}  p{1cm}  |}
\hline
 &\multicolumn{1}{r|}{$X_1$}& & & \multicolumn{4}{ l|}{$X_2$}& & & \multicolumn{4}{ l|}{$X_3$}& & & \multicolumn{4}{ l|}{$X_4$}&  $\cdots$\\
\hline
 & \multicolumn{2}{l|}{$Y_1$}& & & \multicolumn{4}{ l|}{$Y_2$}& & & \multicolumn{4}{ l|}{$Y_3$}& & & \multicolumn{4}{ l|}{$Y_4$}  &$\cdots$\\
\hline
\end{tabular}
\end{center}

There are two floors in this figure.  By Remark \ref{infinitesum}, the concatenation of $X_1\ast X_2\ast \cdots$ ($Y_1\ast Y_2\ast \cdots$) is  $(x_n)$  ($(y_n)$), and we can define the summation of  the concatenated infinite blocks. 

 We call the top floor (bottom floor) the $x$-floor ($y$-floor).  In other words, in the $x$-floor the concatenation of each block $X_i$ is the coding of $x$.  We shall use this diagram representing the blocks in the proofs of some lemmas. Let $(a_n)_{n=1}^{\infty}$ be a coding of some point $x+y\in K_1+K_2$, i.e., $(a_n)=(x_n+y_n)$, where $(x_n)^{\infty}_{n=1}$ and $(y_n)^{\infty}_{n=1}$ are the codings of $x\in K_1$ and $y\in K_2$ respectively. Given $k>0$,  we say $(c_{i_1}c_{i_2}\cdots c_{i_k})$ is a segment of $(a_{i})_{i=1}^{\infty}$  with length $k$ if  there exists $j>0$ such that  $c_{i_1}c_{i_2}\cdots c_{i_k}=a_{j+1}\cdots a_{j+k}$. We define
 \begin{align*}
C=\Big\{(a_n)=(x_n+y_n): &\textrm{ there exists  } \,N\in \mathbb{N}^{+}\,\textrm{such that any segment of}\\
& (a_{N+i})_{i=1}^{\infty} \textrm{ is not a Matching}\Big\}.
\end{align*}
\setcounter{Lemma}{17}
\begin{Lemma}\label{app}
Let $(a_n)\in C$, for any $\epsilon>0$ we can find a coding $(b_n)_{n=1}^{\infty}$ which is the concatenation of infinite Matchings such that
\[|(a_n)_{\beta}-(b_n)_{\beta} |< \epsilon.\]
\end{Lemma}

\begin{proof}
Let $(a_n)\in C$ and $\epsilon>0$, then there exists $n_0\in \mathbb{N}$ satisfying $\beta^{-n_0}< \epsilon$. Now, we choose $(b_n)_{n=1}^{\infty}$ such that its value in base $\beta$ is a point of $E$. Let $b_1b_2b_3\cdots b_{n_0}= a_1a_2a_3\cdots a_{n_0}$. If $a_1a_2a_3\cdots a_{n_0}$ is a Matching or a concatenation of some Matchings, then we can choose arbitrary tail $(b_{n_0+i})_{i=1}^{\infty}$ which is the concatenation of infinite  Matchings. Subsequently we have that
\[
|(a_n)_{\beta}-(b_n)_{\beta} |= |(a_{n_0+1}a_{n_0+2}a_{n_0+3}\cdots)_{\beta}-(b_{n_0+1}b_{n_0+2}b_{n_0+3}\cdots)_{\beta}|\leq M \sum_{i=n_0+1}^{\infty}\beta^{-i}<M(\beta-1)^{-1}\epsilon,
\]
where $M$ is a positive constant which depends  on $\beta$ and    the translations of the IFS's of $K_1$ and $K_2$.  
Hence we prove that there exists a point $ b\in E$, i.e. $b=(b_n)_{\beta}$, such that
\[|(a_n)_{\beta}-(b_n)_{\beta} |< \epsilon.\]
If $a_1a_2a_3\cdots a_{n_0}$ is not a concatenation of some Matchings,  by  virtue of the definition of $(a_n)$, 
$(a_n)=(x_n+y_n)$, where $(x_n), (y_n)$ are the codings of some  points in $K_1$ and $K_2$, respectively. 
However,  $(x_n)$ ($(y_n)$) can be decomposed the concatenation of $X_1\ast X_2\ast \cdots$ ($Y_1\ast Y_2\ast \cdots$).  We use the  following diagram to represent this. 

\begin{center}
\begin{tabular}{ |p{1cm}  p{1cm}  p{1cm}  p{1cm} p{1cm}  p{1cm} p{1cm}  p{1cm} p{1cm}  p{1cm} p{1cm}  p{1cm} p{1cm}  p{1cm} p{1cm}  p{1cm} p{1cm}  p{1cm}  p{1cm}  p{1cm} p{1cm}  p{1cm} p{1cm}  p{1cm} p{1cm}  p{1cm} p{1cm}  p{1cm}  |}
\hline
 &\multicolumn{1}{r|}{$X_1$}& & & \multicolumn{4}{ l|}{$X_2$}& & & \multicolumn{4}{ l|}{$X_3$}& & & \multicolumn{4}{ l|}{$X_4$}&  $\cdots$\\
\hline
 & \multicolumn{2}{l|}{$Y_1$}& & & \multicolumn{4}{ l|}{$Y_2$}& & & \multicolumn{4}{ l|}{$Y_3$}& & & \multicolumn{4}{ l|}{$Y_4$}  &$\cdots$\\
\hline
\end{tabular}
\end{center}
From this figure, we know  that the summation of  $X_1\ast X_2\ast \cdots$ and $Y_1\ast Y_2\ast \cdots$
is precisely the coding $(a_n)$.  Suppose that there exist $p, q$ such that $a_1a_2a_3\cdots a_{n_0}$
 is a prefix of $(X_1\ast X_2\ast \cdots \ast X_p)+(Y_1\ast Y_2\ast \cdots \ast Y_q)$, here we should emphasize that the lengths of $X_1\ast X_2\ast \cdots \ast X_p$ and $Y_1\ast Y_2\ast \cdots \ast Y_q$ may not coincide. However, we can still define the summation of their prefixes. Since $X_1\ast X_2\ast \cdots \ast X_p$ and $Y_1\ast Y_2\ast \cdots \ast Y_q$ do not have the same length,  we assume that $\sum_{i=1}^{p}|X_i|< \sum_{i=1}^{q}|Y_i|$, where $|X_i|$ denotes the length of the block $X_i$, then the first $n_0$ digits of  the ``summation" $(X_1\ast X_2\ast \cdots \ast X_p)+(Y_1\ast Y_2\ast \cdots \ast Y_q)$ is $a_1a_2a_3\cdots a_{n_0}$. Let $k_1=\sum_{i=1}^{p}|X_i| $ and $k_2=\sum_{i=1}^{q}|Y_i|$. Then $(X_1\ast X_2\ast \cdots \ast X_p)^{k_2}+(Y_1\ast Y_2\ast  \cdots \ast Y_q)^{k_1}$ is a Matching or a concatenation of some Matchings as $(X_1\ast X_2\ast \cdots \ast X_p)^{k_2}$ and $(Y_1\ast Y_2\ast  \cdots \ast Y_q)^{k_1}$ have the same length.  Moreover, the initial $n_0$ digits of  $(X_1\ast X_2\ast \cdots \ast X_p)^{k_2}+(Y_1\ast Y_2\ast  \cdots \ast Y_q)^{k_1}$ is $a_1a_2a_3\cdots a_{n_0}$. Now the remaining proof is the same as the first case.
\end{proof}
\setcounter{Remark}{18}
\begin{Remark}
The main idea of this lemma is that any $(a_n)\in C$ can be approximated by a sequence $(c_n)$ which is the concatenation of infinite Matchings. 
\end{Remark}
\setcounter{Lemma}{19}
\begin{Lemma}\label{Closure}
$\overline{E}=K_1+K_2$.
\end{Lemma}
\begin{proof} For every $ \epsilon>0$ and $
x+y\in K_1+K_2$,  we can find a coding $(a_n)$ satisfying $x+y=\sum\limits_{n=1}^{\infty}a_n\beta^{-n}$. If there exists a subsequence of integer $n_{k}\to \infty$ such that $(a_1,a_2,a_3, \cdots, a_{n_k})$ is   a concatenation of some Matchings,  then by the definition of $E\triangleq
\bigcup\limits_{\{\phi_n\}\in\Phi^{\infty}}\bigcap\limits_{n=1}^{\infty}\phi_{1}\circ\phi_{2}\cdots
\phi_{n}([0,1])$ we have $x+y\in E$.  If  $(a_n)\in C$,  by Lemma \ref{app} there
exists $b\in E$ such that $|b-x-y|<\epsilon$.
\end{proof}
\begin{Lemma}\label{IIFS}
$\overline{\bigcup\limits_{i\in \mathbb{N}}\phi_{i}(K_1+K_2)}=K_1+K_2$.
\end{Lemma}
\begin{proof}
 On the one hand, $E=\bigcup\limits_{i\in \mathbb{N}}\phi_{i}(E)$, this equality implies that \[
\overline{E}=\overline{\bigcup\limits_{i\in \mathbb{N}}\phi_{i}(E)}=\overline{\overline{\bigcup\limits_{i\in \mathbb{N}}\phi_{i}(E)}}\supseteq \overline{\bigcup\limits_{i\in \mathbb{N}}\overline{\phi_{i}(E)} }=\overline{\bigcup\limits_{i\in \mathbb{N}}\phi_{i}(K_1+K_2)}, \] i.e. we have
\[\overline{\bigcup\limits_{i\in \mathbb{N}}\phi_{i}(K_1+K_2)} \subseteq K_1+K_2.\]
 On the other hand,
$E=\bigcup\limits_{i\in \mathbb{N}}\phi_{i}(E)\subseteq \bigcup\limits_{i\in \mathbb{N}}\phi_{i}(K_1+K_2)$, therefore we prove the converse inclusion in terms of Lemma \ref{Closure}.
\end{proof}
\begin{proof}[Proof of Theorem \ref{Structure}:]
 Using  Lemma \ref{GenerateMatchings}, we know that there are at most countably many Matchings generated by $D_1$ and $D_2$. If the cardinality of Matchings is infinitely  countable, then by Lemma \ref{IIFS}, $K_1+K_2$ is an attractor of $\Phi^{\infty}$. If the cardinality is finite, then $K_1+K_2$ is a self-similar set. The proof is  similar with Lemmas \ref{Closure} and \ref{app}. The only difference is that it is not necessary to approximate  the coding of $x+y \in K_1+K_2$. In fact, we can  directly find  a coding which is the concatenation of  infinite Matchings such that the value of this infinite coding is $x+y$. In other words,  we have $E=K_1+K_2$. 
\end{proof}

Now, we can prove Corollary \ref{PB}.
When the IFS's of $F_1$ and $F_2$ satisfy the irrationality assumption, it is easy to prove Corollary \ref{PB} due to Peres and Shmerkin \cite{PS}. In fact, we can prove a stronger result. Let us recall their main result.
\setcounter{Theorem}{21}
\begin{Theorem}\label{irrationalityassumption}
Let $F_1$ and $F_2$ be the attractors of  $\{
r_{i}x+a_{i}\}^{n}_{i=1}$, $\{ r^{'}_{j}x+b_j\}^{m}_{j=1}$
respectively. If there exist $i,\,j$ such that
$\frac{\log{r_i}}{\log{r_{j}^{'}}}\notin \mathbb{Q}$, then
$\dim_{H}(F_1+F_2)=\min\{\dim_{H}F_1+\dim_{H}F_2,\,1\}$.
\end{Theorem}

\begin{proof}[Proof of Corollary \ref{PB}]
Firstly, we prove under the irrationality assumption that $$\dim_{H}(F_1+F_2)=\dim_{P}(F_1+F_2)=\dim_{B}(F_1+F_2)=\min\{\dim_{H}F_1+\dim_{H}F_2,\,1\}.$$ Using the theorem above,  if
$\dim_{H}(F_1+F_2)=1$, then \[1=\dim_{H}(F_1+F_2)\leq
\dim_{P}(F_1+F_2)\leq \overline{\dim}_{B}(F_1+F_2)\leq1.\] Suppose
$\dim_{H}(F_1+F_2)=\dim_{H}(F_1)+\dim_{H}(F_2)$. We note that for any $A,\,B\subseteq \mathbb{R}$, we have
$B-A=P_{\frac{\pi}{4}}(A\times B)$, where $P_{\frac{\pi}{4}}(A\times
B)$ denotes the projection of $A\times B$ on the $y$ axis along lines having $45^{\circ}$ angle with the $x$ axis.
Therefore,
\begin{eqnarray*}
\dim_{H}(F_1+F_2)&\leq& \overline{\dim}_{B}(F_1+F_2)\\&\leq
&\overline{\dim}_{B}((-F_2)\times F_1 )\\&\leq&
\overline{\dim}_{B}(F_1)+
\overline{\dim}_{B}(F_2)\\&=&\dim_{H}(F_1)+\dim_{H}(F_2)
\end{eqnarray*}
The second inequality holds as the projection is a Lipschitz map, the third inequality is due to the property of product of fractal sets, see the product formula 7.5, page 102, \cite{FG}.
 For the last equality,  we use the fact that for any self-similar set, its Hausdorff
dimension and the Box dimension coincide.

 If $K_1$ and $K_2$ are generated by the similitudes of $S$ and the cardinality of Matchings is infinitely countable, then we have $\dim_{P}(K_1+K_2)=\overline{\dim}_{B}(K_1+K_2)=\dim_{P}(E)=\overline{\dim}_{B}(E) $ due to Lemma \ref{Closure} and Theorem 3.1 from \cite{MRD}. By Theorem \ref{Structure}, we know that $K_1+K_2$ is a self-similar set if the cardinality of Matchings is finite. Hence, whether the irrationality assumption holds or not we always have $\dim_{P}(K_1+K_2)=\overline{\dim}_{B}(K_1+K_2)$.

 \end{proof}

\section{Dimension of $K_1+K_2$}

\subsection{IFS case}
Let $\sharp D$ be the cardinality of all Matchings generated by $D_1$ and $D_2$. 
In this section we  give a  necessary and sufficient  condition for the finiteness of  $\sharp D$. We know that  $K_1+K_2$ is a self-similar set if $\sharp D$  is finite. Hence, in this case we may make use of various techniques finding the Hausdorff dimension of $K_1+K_2$. 

We say that $D_i, 1\leq i\leq 2$, is homogeneous if the length of all the blocks is equal. For simplicity we may identify the blocks with the lengths of the blocks. There is one point we should keep in mind, namely different blocks of $D_i$ may have the same length. Hence we should count the multiplicity when some blocks have the same length, see the following example.

\begin{Example}
Let $\{f_{1}(x)=\frac{x}{3}
,\,f_{2}(x)=\frac{x+2}{3}\}$ be the IFS of  $K_1$,  $K_2$
 is generated by $\{g_{1}(x)=\frac{x}{3} ,\,g_{2}(x)=\frac{x+8}{9}\}$. The digit sets are $D_1=\{(0),\,(2)\}$ and $D_{2}=\{(0),\,(22)\}$. We can denote $D_1$ by $D_1^{'}=\{1,1\}$. For simplicity we still use $D_1$. Similarly, $D_2=\{1,2\}$. It is clear that $D_1$ is homogeneous and that  two $1$'s in the set refer to  different similitudes.
\end{Example}
It is easy to find that the digits in $D_i$  stand for the length of the blocks  and the similarity ratios, see the following example. 
\begin{Example}\label{dd}
Let $\{f_{1}(x)=\frac{x}{\beta^6}+a_1
,\,f_{2}(x)=\frac{x}{\beta^{10}}+a_2\}$ be the IFS of  $K_1$. We know that $D_1=\{6,10\}$. 6 represents the length of the  block $(0 0 0 0 0 \, (a_1\beta^6))$ and stands for the similarity ratios $\frac{1}{\beta^6}$. 

For this example, by the definition of $K_1$ we have $K_1=f_1(K_1)\cup f_2(K_1)$. Iterating this equation, then we have that $$K_1=f_1\circ f_1(K_1)\cup f_1 \circ f_2(K_1)\cup f_2 \circ f_1(K_1)\cup f_2 \circ f_2(K_1).$$
Hence we obtain  $4$ similitudes  $\{f_1\circ f_1, \,f_1\circ f_2,\, f_1\circ f_2, \,f_2\circ f_2\} $. Their associated digit set  which consists of some blocks can also be denoted by a simpler set $D^{''}=\{12, 16, 16, 20\}$. Similarly, we can  iterate the original IFS for any finite times. For the sake of convenience, we still  use the set of the lengths of the blocks as it  not only stands  for the new iterated blocks but also refers to the  similarity ratios under new  IFS.
\end{Example} 
\setcounter{Definition}{2}
\begin{Definition}
Let $D_1=\{k,k,\cdots, k\}$ be a homogeneous set with $l$ digits. We say $D_2$ is a multiplier set of $D_1$ if we iterate the IFS of $K_2$ for finite times,  all the  numbers of the new digit set $D^{'}$ are the multiplers of $k$,i.e., $D^{'}=\{l_1k, l_2k, \cdots, l_t k\}$, where $l_i\in \mathbb{N}^{+}$. Similarly, if $D_2$ is homogeneous, we can also define $D_1$ as the multiplier set of $D_2$ if $D_1$ satisfies similar property. 
\end{Definition}
\setcounter{Theorem}{3}
\begin{Theorem}\label{finitecc}
$\sharp D$ is finite if and only if $D_1$ ($D_2$) is homogeneous and $D_2$ is a  multiplier set of $D_1$ ($D_1$ is a  multiplier set of $D_2$). 
\end{Theorem}
We partition the proof of this theorem into several lemmas. 
\setcounter{Lemma}{4}
\begin{Lemma}
If $D_1$ is homogeneous and $D_2$ is a  multiplier set of $D_1$,  then $\sharp D$ is finite. 
\end{Lemma}
\begin{proof}
 Let $D_1=\{k,k,\cdots, k\}$  be a  homogeneous  set and $D_2$ be a  multiplier set of $D_1$. By the definition of multiplier set, after finite iterations of the IFS of $K_2$, say $t$ times,  $D_2^{'}= \{l_1k, l_2k, \cdots, l_mk\}$, where $l_i\in \mathbb{N}^{+}$. Now we prove that $\sharp D$ is finite. Let   $D_2=\{s_1,s_2,\cdots s_p\}$, where $s_p\in \mathbb{N}^{+}$.  If we take any $t$ digits  from $D_2$,  each time we can pick any numbers, which means we can pick $s_i$ for any $1\leq k\leq t$ times, then  by the definition of  multiplier set, $s_{i_1}+s_{i_2}+\cdots+s_{i_t}$ is a multiplier of $k$. Since $D_1=\{k,k,\cdots, k\}$  is   homogeneous and  the cardinality of $D_2^{'}= \{l_1k, l_2k, \cdots, l_mk\}$ is finite, it follows that $\sharp D$ is finite. 
\end{proof}
\begin{Lemma}
If $\sharp D$ is finite, then either $D_1$ or $D_2$ is homogeneous.
\end{Lemma}
\begin{proof}
We have proved that if $\sharp D$ is finite, then  $K_1+K_2$ is a self-similar set. This fact implies that for any coding of  $x+y\in K_1+K_2$, say $(a_n)=(x_n+y_n)$,  its associated    value in base $\beta$ is $x+y$, where $(x_n)$ and $(y_n)$ are the codings of $x$ and $y$ respectively. Moreover,  $(a_n)$ is the infinite concatenation of some Matchings.  In other words,   there exists a sequence $N_k\to \infty$ such that  $(a_1a_2\cdots a_{N_k})$  is a concatenation of  some Matchings.  

If neither $D_1$ nor $D_2$ is homogeneous, we may find a coding of some point in  $K_1+K_2$  which does not contain  any Matchings in its arbitrary long  prefix. This contradicts with the assumption  that $\sharp D$ is finite. 

Now we find a coding which satisfies the property we mentioned above. Without loss of generality, we assume that  $D_1=\{a_1,a_2,\cdots, a_p\}$  and  $D_2=\{b_1,b_2,\cdots, b_q\}$, where $a_1\neq a_2$ and $b_1\neq b_2$. 

We demonstrate how we can construct  the coding we need. Recall the definition of $x$-floor and $y$-floor, we know that summation of the concatenation of  the blocks of $x$-floor and $y$-floor is the coding of some point of $K_1+K_2$. Since $a_1\neq a_2$ and $b_1\neq b_2$, we may suppose $a_1\neq b_1$ and put them in the $x$-floor and $y$-floor respectively, see the following figure
\begin{center}
\begin{tabular}{ |p{1cm}  p{1cm}  p{1cm}  p{1cm} p{1cm}  p{1cm} p{1cm}  p{1cm} p{1cm}   p{1cm} p{1cm}  p{1cm} p{1cm}  p{1cm}  p{1cm}  p{1cm} p{1cm}  p{1cm} p{1cm}  p{1cm} p{1cm}  p{1cm} p{1cm}  p{1cm}  |}
\hline
 &\multicolumn{1}{r|}{$a_1$}& & & & & & $\cdots$\\
\hline
 & \multicolumn{2}{l|}{$b_1$}& & & & &  &$\cdots$\\
\hline
\end{tabular}
\end{center}
Here we identify the block with its length. 
Since $a_1\neq b_1$, it follows that no Matching appears. Next, for the $x$-floor, we pick  $a_2$ which satisfies that $a_1+a_2\neq b_1$. If $a_1+a_2=b_1$, then we pick $a_1$ again.  The Matching cannot appear as $a_1\neq a_2$ and $a_1+a_2=b_1$ imply that $a_1+a_1\neq b_1$. Now the $x$-floor and $y$-floor become the following: 
\begin{center}
\begin{tabular}{ |p{1cm}  p{1cm}  p{1cm}  p{1cm} p{1cm}  p{1cm} p{1cm}  p{1cm} p{1cm}   p{1cm} p{1cm}  p{1cm} p{1cm}  p{1cm}  p{1cm}  p{1cm} p{1cm}  p{1cm} p{1cm}  p{1cm} p{1cm}  p{1cm} p{1cm}  p{1cm}  |}
\hline
 &\multicolumn{1}{r|}{$a_1$}& &&&\multicolumn{1}{r|}{$a_2$}& & &  $\cdots$\\
\hline
 & \multicolumn{2}{l|}{$b_1$}& & & & &    &$\cdots$\\
\hline
\end{tabular}
\end{center}
For the $y$-floor, we repeat the same procedure.  Finally we have 
\begin{center}
\begin{tabular}{ |p{1cm}  p{1cm}  p{1cm}  p{1cm} p{1cm}  p{1cm} p{1cm}  p{1cm} p{1cm}  p{1cm} p{1cm}  p{1cm} p{1cm}  p{1cm} p{1cm}  p{1cm} p{1cm}  p{1cm}  p{1cm}  p{1cm} p{1cm}  p{1cm} p{1cm}  p{1cm} p{1cm}  p{1cm} p{1cm}  p{1cm}  |}
\hline
 &\multicolumn{1}{r|}{$a_1$}& & & \multicolumn{4}{ l|}{$a_2$}& & & \multicolumn{4}{ l|}{$a_{i_3}$}& & & \multicolumn{4}{ l|}{$a_{i_4}$}&  $\cdots$\\
\hline
 & \multicolumn{2}{l|}{$b_1$}& & & \multicolumn{4}{ l|}{$b_{i_2}$}& & & \multicolumn{4}{ l|}{$b_{i_3}$}& & & \multicolumn{4}{ l|}{$b_{i_4}$}  &$\cdots$\\
\hline
\end{tabular}
\end{center}
For each step, the Matching does not appear as  the length of the concatenations of blocks from $x$ and $y$-floor are not matched. 
The summation of the infinite  concatenated blocks from $x$ and $y$-floor is the coding we need. 
\end{proof}
Now we may set $D_1=\{k,k,\cdots,k\}$, if $D_2$ is not a multiplier set of $D_1$, we implement similar idea  constructing a coding such that   its arbitrary long prefix is not a concatenation of some Matchings.  

Hence,  in order to  prove Theorem \ref{finitecc}, it remains to prove following lemma. 
\begin{Lemma}\label{inverse1}
Let $D_1=\{k,k,\cdots,k\}$, if $D_2$ is not a multiplier set of $D_1$, then $\sharp D$ is  not finite. 
\end{Lemma}

\begin{proof}
If $\sharp D$ is  finite, then any coding of  $x+y\in K_1+K_2$, say $(a_n)=(x_n+y_n)$,  is the infinite concatenation of some Matchings.  Namely  there exists a sequence $N_k\to \infty$ such that  $(a_1a_2\cdots a_{N_k})$  is a concatenation of  some Matchings.  If we can find a coding $(a_n)$ such that for any $n$ $(a_1a_2\cdots a_{n})$ is not a concatenation of  some Matchings, then we prove this lemma. 
Since $D_2$ is not a multiplier set of $D_1$, it follows that for any finite iterations of the IFS of $K_2$, there always  exists one   block  which is the concatenation of some blocks from $D_2$ such that 
 its length is not a multiplier  of $k$. We let this  block be $(b_1b_2\cdots b_t)$, see the following figure:
\begin{center}
\begin{tabular}{ |p{1cm}  p{1cm}  p{1cm}  p{1cm} p{1cm}  p{1cm} p{1cm}  p{1cm} p{1cm}  p{1cm} p{1cm}  p{1cm} p{1cm}  p{1cm} p{1cm}  p{1cm} p{1cm}  p{1cm}  p{1cm}  p{1cm} p{1cm}  p{1cm} p{1cm}  p{1cm} p{1cm}  p{1cm} p{1cm}  p{1cm}  |}
\hline
 & \multicolumn{2}{l|}{$k$}& & & \multicolumn{4}{ l|}{$k$}& & & \multicolumn{4}{ l|}{$k$}& & & \multicolumn{4}{ l|}{$k$}  &$\cdots$\\
\hline
 &\multicolumn{1}{r|}{$Y_1$}& & & \multicolumn{4}{ l|}{$Y_2$}& & & \multicolumn{4}{ l|}{$\cdots$}& & & \multicolumn{4}{ l|}{$Y_{N}$}&  $\cdots$\\
\hline
\end{tabular}
\end{center}
We may assume that $(b_1b_2\cdots b_t)= Y_1\ast Y_2\ast \cdots \ast Y_N$ for some $N$, where each $Y_i$ is some block from $D_2$. By the assumption we know that its length is not a multiplier of $k$. 
 Hence we can  find such coding $(a_n)$ (sum of the $x$ and $y$-floor) satisfying that  for any $n$,  $(a_1a_2\cdots a_{n})$ is not a concatenation of  some Matchings. 
\end{proof}
\setcounter{Remark}{7}
\begin{Remark}
When $K_1+K_2$ is a self-similar set, we do not know whether $\sharp D$ is finite or not. 
\end{Remark}
If  $\sharp D$ is finite, then $K_1+K_2$ is a self-similar set.  In this case, we can explicitly find all the similitudes of the IFS. Therefore we can implement many ideas calculating $\dim_{H}(K_1+K_2)$. We do not discuss this problem in detail. 
\subsection{IIFS case}
Comparing with IFS case, it is much more complicated  when $K_1+K_2$ is  a unique attractor of some IIFS. 
We have mentioned the main reasons in the second  section.  

By Lemma \ref{Closure}, we know that when $\sharp D$ is infinitely countable, $\overline{E}=K_1+K_2$. If $\overline{E}\setminus E$  is uncountable, we may not  calculate the dimension of  $K_1+K_2$ in terms of the dimensional theory of IIFS. Hence, we need to find some class that can guarantee $\dim_{H}(E)=\dim_{H}(K_1+K_2)$. In fact,  even for calculating $\dim_{H}(E)$,  it is not easy to find $\dim_{H}(E)$ when the IIFS has some overlaps \cite{NT,HM}. 

Let $(a_n)_{n=1}^{\infty}$ be the coding of some point $x+y\in K_1+K_2$, i.e., $(a_n)=(x_n+y_n)$, where $(x_n)$ and $(y_n)$  are the codings of $x$ and $y$ respectively.
Recall the definition of  $C$,
\[C=\{(a_n):\mbox{there exists }\,N\in \mathbb{N}^{+}\,\mbox{such that any  segment of }(a_{N+i})_{i=1}^{\infty}\mbox{ is  not   a Matching}\}.\]
We have
\setcounter{Lemma}{8}
\begin{Lemma}\label{almostequal}
If $C$ is countable,  then we have that $E=K_1+K_2$ apart from a countable set.
\end{Lemma}
\begin{proof}
By Lemma \ref{Closure}, $\overline{E}=K_1+K_2$.  It remains  to prove that there are only countably many limit points of $E$ which are not in $E$.  For any $x+y\in K_1+K_2=\overline{E}$, there is a coding $(a_n)$ such that the value of this coding is $x+y$. If there exists $n_k\to \infty$ satisfying that  $(a_1a_2 \cdots a_{n_k})$ is a Matching or a concatenation of some Matchings, by the definition of $E\triangleq
\bigcup\limits_{\{\phi_n\}\in\Phi^{\infty}}\bigcap\limits_{n=1}^{\infty}\phi_{1}\circ\phi_{2}\circ\cdots\circ
\phi_{n}([0,1])$, we know that $x+y\in E$. If  $(a_n)\in C$,  then  $\overline{E}\setminus E$ is countable as  $C$ and the cardinality of all the Matchings are  countable.
\end{proof}
The following lemma gives a sufficient condition which implies that  $C$ is countable.
\begin{Lemma}\label{Countablecriterion}
$C$ is countable if   there exists $k$ such that $D_1=\{k,k,\cdots,k,2k\}$ and  $D_2=\{k,k,\cdots,k,2k\}$, i.e. both  $D_1$ and $D_2$  have only  blocks with length $k$ apart from the last block with length $2k$.
\end{Lemma}
\begin{proof}
If  $D_1=\{k,k,\cdots,k,2k\}$ and  $D_2=\{k,k,\cdots,k,2k\}$, we need to find all  possible sequences of $C$. Without loss of generality, we assume that the prefix of the summation of the $x$ and $y$-floor does not contain any Matchings. Firstly, we choose two blocks from $D_1$ and $D_2$ and stack on the $x$-floor and $y$-floor respectively. We can pick only $k$ from $D_1$ and $2k$ from $D_2$(or $2k$ from $D_1$ and $k$ from $D_2$). Otherwise, a Matching will appear, see the following figure
 \begin{center}
\begin{tabular}{ |p{1cm}  p{1cm}  p{1cm}  p{1cm} p{1cm}  p{1cm} p{1cm}  p{1cm} p{1cm}  p{1cm} p{1cm}  p{1cm} p{1cm}  p{1cm} p{1cm}  p{1cm} p{1cm}  p{1cm}  p{1cm}  p{1cm} p{1cm}  p{1cm} p{1cm}  p{1cm} p{1cm}  p{1cm} p{1cm}  p{1cm}  |}
\hline
 &\multicolumn{1}{r|}{$k$}& &  $\cdots$\\
\hline
 & \multicolumn{2}{l|}{$2k$}& &$\cdots$\\
\hline
\end{tabular}
\end{center}
Then at the second step for the $x$-floor  we cannot take any block of $D_1$ with length $k$ as $k+k=2k$ and  a new Matching  appears.  Hence for the $x$-floor we can pick only the block with length $2k$. 
Similarly, for the $y$-floor  we cannot take a block of $D_2$ with length $k$ as $k+2k=2k+k$, which can generate a new Matching. Therefore, we must take a block with length $2k$ for the $y$-floor if we do not want  a new 	Matching  to appear.  The figure now is
  \begin{center}
\begin{tabular}{ |p{1cm}  p{1cm}  p{1cm}  p{1cm} p{1cm}  p{1cm} p{1cm}  p{1cm} p{1cm}  p{1cm} p{1cm}  p{1cm} p{1cm}  p{1cm} p{1cm}  p{1cm} p{1cm}  p{1cm}  p{1cm}  p{1cm} p{1cm}  p{1cm} p{1cm}  p{1cm} p{1cm}  p{1cm} p{1cm}  p{1cm}  |}
\hline
 &\multicolumn{1}{r|}{$k$}& &\multicolumn{1}{r|}{$2k$}& $\cdots$\\
\hline
 & \multicolumn{2}{l|}{$2k$}& & \multicolumn{1}{r|}{$2k$} &$\cdots$\\
\hline
\end{tabular}
\end{center}
It is easy to see that if we  want to avoid  the new Matchings  in the summed blocks of two floors we cannot choose blocks freely from the second step on. The figure below  illustrates this idea.
\begin{center}
\begin{tabular}{ |p{1cm}  p{1cm}  p{1cm}  p{1cm} p{1cm}  p{1cm} p{1cm}  p{1cm} p{1cm}  p{1cm} p{1cm}  p{1cm} p{1cm}  p{1cm} p{1cm}  p{1cm} p{1cm}  p{1cm}  p{1cm}  p{1cm} p{1cm}  p{1cm} p{1cm}  p{1cm} p{1cm}  p{1cm} p{1cm}  p{1cm}  |}
\hline
 &\multicolumn{1}{r|}{$k$}& & & \multicolumn{4}{ l|}{$2k$}& & & \multicolumn{4}{ l|}{$2k$}& & & \multicolumn{4}{ l|}{$2k$}&  $\cdots$\\
\hline
 & \multicolumn{2}{l|}{$2k$}& & & \multicolumn{4}{ l|}{$2k$}& & & \multicolumn{4}{ l|}{$2k$}& & & \multicolumn{4}{ l|}{$2k$}  &$\cdots$\\
\hline
\end{tabular}
\end{center}
 From the analysis above, we see that the sequences in $C$ are eventually periodic. Thus, we prove that $C$ is countable.
\end{proof}
\setcounter{Remark}{10}
\begin{Remark}
 The condition of the lemma is not necessary, for instance,  let $D_1=\{k, 2k\}$ and  $D_2=\{k,3k\}$. We can similarly prove that in this case $C$ is countable. 
Generally it is not easy to find all the Matchings. However, for the case in this lemma  we can find all possible Matchings without much calculation. 
\end{Remark}
This lemma enables us to define  the following IFS. 

For any $k\in \mathbb{N}^{+}$, let the IFS's of $K_1$ and $K_2$ be
 \begin{equation}\label{IFS1}
 \left\{f_{i}(x)=\dfrac{x}{\beta^k}+a_i, 
1\leq i \leq n-1, f_{n}(x)=\dfrac{x}{\beta^{2k}}+a_{n}\right\} \end{equation}
and 
 \begin{equation}\label{IFS2}
 \left\{g_{j}(x)=\dfrac{x}{\beta^k}+b_j, 
1\leq j \leq n-1, g_{n}(x)=\dfrac{x}{\beta^{2k}}+b_{n}\right\},
 \end{equation}
where $a_i, b_j\in \mathbb{R}^{+}\cup \{0\}$. We denote their attractors by $K_1$ and $K_2$ respectively. Without loss of generality, we let the convex hull of $K_i$ be $[0, B_i]$, $0\leq i\leq 2$. This assumption yields that $f_{i}([0, B_1])\subset [0, B_1]$, $1\leq i\leq n$ and  $g_{j}([0, B_2])\subset [0, B_2]$, $1\leq j\leq n$.

 Let 
$D=\{\hat{R}_{1} ,\,\hat{R}_{2},\,\cdots,\,\hat{R}_{n-1},\,
\hat{R}_{n}\cdots\} $
be all the Matchings  generated by $D_1=\{k,k,\cdots, 2k\}$ and $D_2=\{k,k,\cdots, 2k\}$ and its associated IIFS be 
$\Phi^{\infty}\triangleq\{\phi_1,\,\phi_2,\,\phi_3,\,\phi_4,\cdots\}
$.  Define
$E\triangleq
\bigcup\limits_{\{\phi_n\}\in\Phi^{\infty}}\bigcap\limits_{n=1}^{\infty}\phi_{1}\circ\phi_{2}\circ\cdots\circ
\phi_{n}([0,B_1+B_2])$.  We know that a Matching $\hat{R}_{i}$ is a block. Suppose  $\hat{R}_{i}= (c_1c_2,\cdots, c_p)$ for some $p\in \mathbb{N}$. We call  each $c_i$ the digit of $\hat{R}_{i}$. Since $D_1$ and $D_2$ have a finite number of  blocks, it follows that the range of every possible  digit $c_j$ in each Matching $\hat{R}_{i}$ is finite, i.e. $c_j$ can take only finite numbers. Let $c$ be the   positive constant defined as follows: 
  \[c=\min\{ |c_i-c_j|: c_i\, \mbox{and}\, c_j\, \mbox{are any digits which are from two  Matchings}\}.\] 
Similarly, we let $A$ and $B$ be the largest digits of the blocks of $D_1$ and $D_2$ respectively. 
\setcounter{Theorem}{11}
\begin{Theorem}\label{Dimension}
Let $K_1$ and $K_2$ be two self-similar sets with IFS's  \eqref{IFS1} and \eqref{IFS2}, respectively.
Then $E=K_1+K_2$ up to a countable set. 
If 
 \[
A+ B+B_1+B_2<c(\beta-1),
 \] then $\Phi^{\infty}$ satisfies the open set condition and  $\dim_{H}(K_1+K_2)$ is computable.
\end{Theorem}
\begin{proof}[Proof of  Theorem  \ref{Dimension}]  By Lemmas \ref{almostequal} and \ref{Countablecriterion}, we prove the first statement. For the second statement,
 given any two Matchings $(s_1s_2\cdots s_p),(t_1t_2\cdots t_q)$ with $p<q$, their associated similitudes are $\phi_{s_1s_2\cdots s_p}(x)=\beta^{-p}x+\sum_{i=1}^{p}s_i\beta^{-i}$ and   $\phi_{t_1t_2\cdots t_q}(x)=\beta^{-q}x+\sum_{i=1}^{q}t_i\beta^{-i}$  respectively.  Let $V=(0, B_1+B_2)$, simple calculation implies that
 \[\phi_{s_1s_2\cdots s_p}(V)=\left(\sum_{i=1}^{p}s_i\beta^{-i}, \sum_{i=1}^{p}s_i\beta^{-i}+(B_1+B_2)\beta^{-p}\right)\]
  \[\phi_{t_1t_2\cdots t_q}(V)=\left(\sum_{i=1}^{q}t_i\beta^{-i}, \sum_{i=1}^{q}t_i\beta^{-i}+(B_1+B_2)\beta^{-q}\right).\]
 We assume that  $(s_1s_2\cdots s_p)<(t_1t_2\cdots t_q)$, i.e, there exists $1\leq i_{0}\leq p$ such that $s_k=t_k$ for any $1\leq k\leq i_{0}-1$ and $s_{i_0}<t_{i_0}$. By the definition of  $c$, we can check that the two intervals above do not overlap, namely $\phi_{s_1s_2\cdots s_p}(V)\cap\phi_{t_1t_2\cdots t_q}(V)=\emptyset $. It remains to prove that $\phi(V)\subset V$ for any $\phi\in \Phi^{\infty}$. Let $\phi$ be generated by the Matching $\hat{R}_1\ast \hat{R}_2+\hat{T}_1\ast \hat{T}_2$, the associated similitudes of $\hat{R}_i$ and $\hat{T}_i$ are $H_i(x)$ and $I_{i}(x)$ respectively. Let  the length of $\hat{R}_1\ast \hat{R}_2+\hat{T}_1\ast \hat{T}_2$ be $k_0$.  It is easy to find that \[\phi(x)=H_1\circ H_2(x)+I_1\circ I_2(0). \]
  Hence, \[\phi(V)=\left( H_1\circ H_2(0)+I_1\circ I_2(0), H_1\circ H_2(0)+I_1\circ I_2(0)+\dfrac{B_1+B_2}{\beta^{k_0}}\right).\]
  Recall the assumption of $K_1$ and $ K_2$,  the convex hull of $K_i$ is $[0, B_i]$, $1\leq i\leq 2$, i.e.,  $ H_s([0, B_1])\subset [0, B_1],  1 \leq  s\leq 2 $ and $ I_t([0, B_2])\subset [0, B_2]$, $1 \leq   t\leq 2$.
 Therefore $0<\phi(x)<B_1+B_2$. Similarly, we can prove that $\phi(V)\subset V$ for any $\phi\in \Phi^{\infty}$. As such $\Phi^{\infty}$ satisfies the open set condition.
The calculation of   $\dim_{H}(K_1+K_2)$ now is a straightforward application of Theorem  \ref{DimensionIIFS}.
\end{proof}

Generally we do not know how to calculate $\dim_{P}(K_1+K_2)$ or when do we have following equality
$$\dim_{H}(K_1+K_2)= \dim_{P}(K_1+K_2)=\dim_{B}(K_1+K_2).$$
We finish this section by making some remarks  on these two problems. 
Let $F_n$ be the attractor of the first $n$ similitudes of $\Phi^{\infty}$, i.e., $F_n$ is the attractor of the IFS $\{\phi_i\}_{i=1}^{n}$. Clearly  $$F_1\subset F_2\subset\cdots \subset F_n\subset \cdots.$$
Recall the definition of Hausdorff metric \cite{FG}. 
Given two compact sets $J_1,J_2\subset \mathbb{R}$, then the Hausdorff  metric of $J_1$ and $J_2$ is defined by 
$$\mathcal{H}(J_1, J_2)=\inf\{s: J_1 \subset (J_2)_s, J_2 \subset (J_1)_s\},$$
where $(A)_s=\{x: \mbox{there exists } y \in A\, \mbox{such that }\,|x-y|\leq s\}$.

We have 
\setcounter{Lemma}{12}
\begin{Lemma}\label{appro}
$\overline{\cup_{n=1}^{\infty}F_n}=K_1+K_2$. 
\end{Lemma}
\begin{proof}
$0\leq \mathcal{H}( \overline{\cup_{n=1}^{\infty}F_n}, K_1+K_2)\leq \mathcal{H}( F_n, K_1+K_2)\to 0$ as $n\to \infty$. 
Here $\mathcal{H}( F_n, K_1+K_2)\to 0$  can be found in \cite{HF}.
\end{proof}
\setcounter{Proposition}{13}
\begin{Proposition}
If $(\overline{\cup_{n=1}^{\infty}F_n})\setminus (\cup_{n=1}^{\infty}F_n) $ is  a countable set, then 
$$\dim_{H}(K_1+K_2)= \dim_{P}(K_1+K_2)=\dim_{B}(K_1+K_2).$$
\end{Proposition}
\begin{proof}
Since $(\overline{\cup_{n=1}^{\infty}F_n})\setminus (\cup_{n=1}^{\infty}F_n) $ is countable, it follows by Lemma \ref{appro} that 
\begin{eqnarray*}
\dim_{P}(K_1+K_2)&=&\dim_{P}(\cup_{n=1}^{\infty}F_n)=\lim_{n\to \infty}\dim_{P}(F_n)\\&=&\dim_{H}(\cup_{n=1}^{\infty}F_n)= \dim_{H}(\overline{\cup_{n=1}^{\infty}F_n})=\dim_{H}(K_1+K_2).
\end{eqnarray*}
We finish the proof by Corollary \ref{PB}.
\end{proof}
\section{Examples}
In this section, we give some examples for which Theorem \ref{irrationalityassumption} cannot calculate $\dim_{H}(K_1+K_2)$.
\begin{Example}
Let $K_1=K_2$ be the self-similar sets  with  IFS  $\{g_{1}(x)=\frac{x}{3}
,\,g_{2}(x)=\frac{x+8}{3^2}\}$, then
$\dim_{H}(K_1+K_2)=\frac{\ln t_{0}}{-\ln 3}$,
where $t_{0}$ is the smallest positive root of $t^3-t^2-2t+1=0$.
\end{Example}
We know that
$D_1=D_{2}=\{(0),\,(22)\}$,  all the Matchings which are generated by
$D_1$ and $D_2$ are
\[D=\{(0),(22),(44),\,(242),\,(2442)
,\,(24442),\,(244442)\cdots\}.\] The corresponding  IIFS of $D$ is
\[\Phi^{\infty}=\{\varphi_{1}=f_0,\,\varphi_{2}=f_2\circ f_2,\,\varphi_{3}=f_4\circ
f_4,\,\varphi_{4}=f_2\circ f_4\circ f_2,\,\cdots \},\] where
$f_0(x)=\frac{x}{3},\,f_2(x)=\frac{x+2}{3},\,f_4(x)=\frac{x+4}{3}$.

By Theorem \ref{Dimension}, $\dim_{H}(K_1+K_2)=\dim_{H}(E)$. Obviously this IIFS
satisfies  the  OSC, i.e. \[\varphi_{i}((0,2))\cap \varphi_{j}((0,2))=\varnothing\] for any $i\neq j$ and $\varphi_{i}((0,2))\subseteq (0,2)$ for any $i\in \mathbb{N}$. Now we can use Theorem  \ref{DimensionIIFS} to calculate the dimension. It is easy to check that 
$$\dim_{H}(K_1+K_2)<\min\{1, \dim_{H}(K_1)+\dim_{H}(K_2)\}.$$
This example illustrates that without the irrationality assumption, the expected dimension of $K_1+K_2$ may not  be achieved. This differs from Peres and Shmerkin's result \cite{PS}. 

\begin{Example}Let $\{f_{1}(x)=\frac{x}{\beta}
,\,f_{2}(x)=\frac{x+2}{\beta}\}$ and
$\{g_{1}(x)=\frac{x}{\beta} ,\,g_{2}(x)=\frac{x}{\beta^2}+\frac{2}{\beta} + \frac{2}{\beta^2}\}$ be the IFS's of $K_1$ and $K_2$
respectively. Then $K_1+K_2$ is a self-similar set, the
IFS is
$\{\varphi_{1}(x)=\frac{x}{\beta},\,\varphi_{2}(x)=\frac{x+2}{\beta},\,\varphi_{3}(x)=\frac{x}{\beta^2}+\frac{2}{\beta}+\frac{4}{\beta^2},\,\varphi_{4}(x)=\frac{x}{\beta}+\frac{4}{\beta}+\frac{2}{\beta^2},\,\varphi_{5}(x)=\frac{x}{\beta^2}+\frac{4}{\beta}+\frac{4}{\beta^2}\}$.
This IFS does not satisfy the OSC generally, in fact it is of  finite type if  $\beta$ is a Pisot number, see  \cite[Theorem 2.5]{NW}. Hence, we can calculate the Hausdorff dimension of $K_1+K_2$ in terms of the main result of  \cite{NW}. We omit the details.
\end{Example}
\section{Final remarks}
The main result of this paper is that $K_1+K_2$ is either a self-similar set or a unique attractor of some IIFS.  However, to calculate the dimension of  $K_1+K_2$  is difficult, especially the IIFS case.  As  in this case, we should consider the limit points of $E$ as well as the separation condition. Ignoring either of them may  hinder the calculation of  the dimension of  $K_1+K_2$.  In fact, even finding all the Matchings is not a trivial task. 
On the other hand, 
we
may implement the  Vitali process if the IIFS has overlaps, see \cite[Theorem 3.1]{Moran}, this process is complicated.  Ngai and Tong  \cite{NT} gave a dimensional formula of $J_0$ under the so-called weak separation condition, but it is still not easy to check this condition generally. Some techniques of \cite{Hochman} are useful to analyze the Hausdorff dimension of self-similar sets.

\section*{Acknowledgements}
The author would like to thank the anonymous referees for many suggestions and remarks, and to Karma Dajani for some suggestions  on the previous versions of the manuscript. 
The work was supported by the China Scholarship Council and  by the National Natural Science Foundation of China no 11271137.

\end{document}